\documentclass{article}
\usepackage[utf8]{inputenc}

\usepackage[normalem]{ulem}

\usepackage{amsmath,amsthm}
\usepackage{amsfonts}
\usepackage{amssymb}
\usepackage{enumitem}
\usepackage{makeidx}
\usepackage{tikz}
\usepackage{tikz-cd}
\usepackage{wrapfig}
\usepackage{xcolor}
\usepackage{mathtools}
\usepackage{graphicx}
\usepackage{caption}
\usepackage{comment}
\usepackage{hyperref}
\usepackage{cleveref}
\usepackage{bm}
\usepackage{authblk}%To add affiliation/email

\makeatletter
\renewcommand{\fnum@figure}{Fig. \thefigure}
\makeatother

\makeatletter
\newcommand{\mylabel}[2]{#2\def\@currentlabel{#2}\label{#1}}
\makeatother

\newtheorem{theorem}{Theorem}[section]
\newtheorem{theorem*}[theorem]{*Theorem}
\newtheorem{prop}[theorem]{Proposition}
\newtheorem{lemma}[theorem]{Lemma}

\newtheorem{cor}[theorem]{Corollary}

\newtheorem{prop*}[theorem]{*Proposition}
\newtheorem{lema*}[theorem]{*Lemma}
\newtheorem{cor*}[theorem]{*Corollary}

\theoremstyle{definition}
\newtheorem{definition}[theorem]{Definition}
\newtheorem{question}[theorem]{Question}
\newtheorem{definition*}[theorem]{*Definition}

\newtheorem{defs*}[theorem]{*Definitions}

\newtheorem{note}[theorem]{Note}
\newtheorem{claim}[theorem]{Claim}
\newtheorem{remark}[theorem]{Remark}

\newtheorem{conjecture}[theorem]{Conjecture}

\title{Gromov-Hausdorff distances from simply connected geodesic spaces to the circle.}
\author[1]{Saúl Rodríguez Martín}
\affil[1]{The Ohio State University, rodriguezmartin.1@osu.edu}
\date{\vspace{-20pt}}

\begin{document}
\maketitle
%\tableofcontents

\begin{abstract}
We prove that the Gromov-Hausdorff distance from the circle with its geodesic metric to any simply connected geodesic space is never smaller than $\frac{\pi}{4}$. We also prove that this bound is tight through the construction of a
simply connected geodesic space $\mathrm{E}$ which attains the lower bound $\frac{\pi}{4}$. We deduce the first statement from a general result that we also establish which gives conditions on how small the Gromov-Hausdorff distance between
two geodesic metric spaces $(X, d_X)$ and $(Y, d_Y )$ has to be in order for
$\pi_1(X)$ and $\pi_1(Y)$ to be isomorphic.
\end{abstract}

\section{Introduction}

In the following, for $n\geq1$ we give the unit sphere $\mathbb{S}^n:=\{z\in\mathbb{R}^{n+1};\|z\|=1\}$ its intrinsic metric $d_{\mathbb{S}^n}$. That is, for any two points $p,q\in\mathbb{S}^n$ seen as vectors in $\mathbb{R}^{n+1}$, $d_{\mathbb{S}^n}(p,q)$ is the angle between $p$ and $q$.

The Gromov-Hausdorff (GH) distance provides a quantitative measure of how far two metric spaces are from being isometric. Since being introduced by Edwards (\cite{Ed}, 1975) and Gromov (\cite{Gr1,Gr2}, 1981), it has proved useful in the study of, for example, shapes formed by point cloud data \cite{MS,BBK}, convergence results for sequences of Riemannian manifolds \cite{CC,PW,C1,C2}, differentiability results in metric measure spaces \cite{Ke,Ch} or the stability of topological invariants of metric spaces under small deformations (see \cite{Pe} and the present article).

However, it is hard to find the exact value of the GH distance between two given metric spaces. 
There have been recent efforts to determine the precise value of the GH distance between spheres. In \cite[Theorem B]{LMS} it is proved that $d_{\textup{GH}}(\mathbb{S}^n,\mathbb{S}^m)\geq\frac{1}{2}\arccos\left(\frac{-1}{m+1}\right)$ for all $n>m$ which implies that 
\begin{equation}
\label{eq:d1n}d_{\textup{GH}}(\mathbb{S}^n,\mathbb{S}^1)\geq\frac{\pi}{3}
\end{equation}
for all $n\geq 2$. Via these lower bounds and matching upper bounds, the authors find the precise value of the GH distance for the pairs $(\mathbb{S}^1,\mathbb{S}^2)$, $(\mathbb{S}^1,\mathbb{S}^3)$, and $(\mathbb{S}^2,\mathbb{S}^3)$.  Similarly, in \cite[Lemma 2.3]{Ka} it was proved that for any interval $I\subseteq\mathbb{R}$ we have $d_{\textup{GH}}(I,\mathbb{S}^1)\geq\frac{\pi}{3}$ (see also \cite[Prop. B.1]{LMS}; in \cite[Theorem 6.8]{JT} the exact GH distances from $\mathbb{S}^1$ to intervals of any length $\lambda\in[0,\infty)$ are computed). In \cite{ABC} the authors identify  lower bounds for $d_{\textup{GH}}(\mathbb{S}^m,\mathbb{S}^n)$ that are often tighter than the ones from \cite{LMS}.

Via considerations related to persistent homology and the filling radius, in \cite[Remark 9.19]{LMO} it was deduced that for any compact geodesic metric space $X$ one has $d_{\textup{GH}}(X,\mathbb{S}^1)\geq\frac{\pi}{6}$ and, as a generalization of Equation (\ref{eq:d1n}),  the following conjecture was formulated:
\begin{conjecture}[{\cite[Conjecture 4]{LMO}}]\label{45dfgty6rrr}
For any compact, simply connected geodesic space $X$ we have $d_{\textup{GH}}(X,\mathbb{S}^1)\geq\frac{\pi}{3}$.
\end{conjecture}

In \Cref{45dfgty6rrr}, the condition that $X$ is geodesic is necessary to relate the metric of $X$ to the fact that it is simply connected; if not, we could consider the space $X=\mathbb{S}^1\setminus\{1\}$, with its metric inherited from  $\mathbb{S}^1$. Then $X$ is simply connected, but we clearly have $d_{\textup{GH}}(X,\mathbb{S}^1)=0$. The present article gives a complete answer to \Cref{45dfgty6rrr}; we state the main results after recalling some standard notation.

Let $( X,d_{ X})$ be a metric space, and let $A,B\subseteq  X$ be nonempty and let $x\in X$. The distance from $x$ to $A$ is defined as $d_{ X}(x,A)=\inf_{a\in A}d_{ X}(x,a)$,
and the \textit{Hausdorff distance} between $A$ and $B$ is given by
\[
d^{ X}_{\text{H}}(A,B)=\max\left(\sup_{a\in A}d_{ X}(a,B),\sup_{b\in B}d_{ X}(b,A)\right).
\]
Given two nonempty metric spaces $( X,d_{ X})$ and $(Y,d_{Y})$, we write $( X,d_{ X})\cong(Y,d_{Y})$, or just $ X\cong Y$ when the metrics are clear, whenever $( X,d_{ X})$ and $(Y,d_{Y})$ are isometric. The \textit{Gromov-Hausdorff (GH) distance} between $ X$ and $Y$ is defined as the value in $[0,\infty]$ given by
\[
d_{\textup{GH}}( X,Y)=
\inf\{d^{Z}_{\text{H}}( X',Y');(Z,d_{Z})\text{ metric space; } X',Y'\subseteq Z; X'\cong  X;Y'\cong Y\}.
\begin{comment}
\footnote{The collection appearing in this equation is too big to be a set, but this issue can be easily fixed (note that we may assume $Z= X'\bigcupY'$).}
\end{comment}
\]

Recall that a metric space $(X,d_X)$ is a \textit{length space} when for all $x,x'\in  X$, the distance $d_X(x,x')$ is the infimum of lengths of paths $\gamma:[0,1]\to X$ with $\gamma(0)=x$ and $\gamma(1)=x'$ (cf. \cite[Chapter 2]{BBI}). If in addition we have that for all $x,x'\in X$ there is a path $\gamma$ from $x$ to $x'$ with length $d_X(x,x')$, we say that $X$ is \textit{geodesic}.

\begin{theorem}\label{Counterexample}
There exists a simply connected geodesic space \textup{E} with $d_{\textup{GH}}(\textup{E},\mathbb{S}^1)=\frac{\pi}{4}$.
\end{theorem}

\begin{theorem}\label{34rf43ewewwe}
Let $ X,Y$ be length spaces. Suppose there is a constant $D>0$ such that $d_{\textup{GH}}( X,Y)<D$ and all loops of diameter $<4D$ are nulhomotopic in $X$ and $Y$. Then $\pi_1( X)$ and $\pi_1(Y)$ are isomorphic.
\end{theorem}

\begin{cor}\label{BigTheorem}
Any simply connected length space X satisfies $d_{\textup{GH}}( X,\mathbb{S}^1)\geq\frac{\pi}{4}$.
\end{cor}

To prove \Cref{BigTheorem} from \Cref{34rf43ewewwe} note that every loop $\alpha:[0,1]\to\mathbb{S}^1$ with diameter $<\pi$ is nulhomotopic. So if there was a simply connected space $X$ with $d_{\textup{GH}}(X,\mathbb{S}^1)<\frac{\pi}{4}$, then we could choose the constant $D$ from \Cref{34rf43ewewwe} to be just below $\frac{\pi}{4}$, which would give a contradiction as $\pi_1(\mathbb{S}^1)\neq0$.\\

\Cref{BigSection} is almost entirely devoted to the proof of \Cref{34rf43ewewwe}. At the end of the section we also give a lower bound for the distances from simply connected length spaces to $\mathbb{S}^1$ with the metric inherited from $\mathbb{R}^2$.

\begin{remark}
As was well pointed out to me by F. Mémoli, the arguments used in \Cref{BigSection} are similar to those in Petersen's article \cite{Pe}. The Theorem in \cite[Section 4]{Pe} claims that, under adequate conditions, if the GH distance between two metric spaces $X,Y$ is small enough, then they are homotopy equivalent; in our case, as we are comparing simply connected spaces with $\mathbb{S}^1$, it will be enough to give conditions on $d_{\textup{GH}}(X,Y)$ such that $\pi_1(X)\cong\pi_1(Y)$.
This allows us to give better bounds for $d_{\textup{GH}}(X,Y)$ for our purposes; for example, in the Main Lemma of \cite[Section 2]{Pe}, if one sets $n=1$ and if $X$ is a length space instead of a LGC$^0(\rho)$ space, then one can improve Petersen's conclusion $\overline{f}(\Delta)\subseteq B(f(v),\rho_1(\varepsilon))$ (with $\rho_1(\varepsilon)\geq2\varepsilon$) to the sharper condition $\overline{f}(\Delta)\subseteq B(f(v),\varepsilon)$.
\end{remark}

\Cref{EO} is devoted to the construction of a simply connected geodesic space E with $d_{\textup{GH}}(\mathrm{E},\mathbb{S}^1)\leq\frac{\pi}{4}$. This, in conjunction with \Cref{BigTheorem}, proves
\Cref{Counterexample}. In \Cref{EO} we will not use the definition of Gromov-Hausdorff distance given above; we now recall (cf. \S7.3 in \cite{BBI}) an equivalent definition based on correspondences between sets, which will be more convenient for our purposes. 

Given two sets $X$ and $Y$, we say a relation $R\subseteq  X\times Y$ is a \textit{correspondence} between $ X$ and $Y$ if $\pi_{ X}(R)= X$ and $\pi_{Y}(R)=Y$, where $\pi_{ X}: X\times Y\to  X$ and $\pi_{Y}: X\times Y\to Y$ are the coordinate projections. If $( X,d_{ X})$ and $(Y,d_{Y})$ are metric spaces, we define the distortion of a nonempty relation $R\subseteq  X\times Y$ as 
\begin{equation}
\text{dis}(R):=\sup\{|d_{ X}(x,x')-d_{Y}(y,y')|;(x,y),(x',y')\in R\}\in[0,\infty].
\end{equation}
In \cite[Theorem 7.3.25]{BBI} it is proved that, if $( X,d_{X}),(Y,d_{Y})$ are nonempty metric spaces, then
\begin{equation}\label{distortandGH}
d_{\textup{GH}}( X,Y)=\frac{1}{2}\inf\{\text{dis}(R);R\subseteq  X\times Y\text{ correspondence between $ X$ and }Y\}.
\end{equation}

In \Cref{UniquenessSection} we adress the question of whether the space $\mathrm{E}$ we construct to prove \Cref{Counterexample}, which is a compact $\mathbb{R}$-tree, is the unique space satisfying \Cref{Counterexample}. It turns out that it is not hard to use E to construct a big family of spaces that satisfy \Cref{Counterexample} (see \Cref{453retrgfgdff}), and that there are $\mathbb{R}$-trees which satisfy \Cref{Counterexample} and are not isometric to E (see \Cref{Note3efrdfd}). 
However, in \Cref{34we23ioepqwpw9} we prove a minimality result for our space E: it is, up to isometry, the only complete $\mathbb{R}$-tree of length at most $\frac{5\pi}{4}$ which satisfies \Cref{Counterexample}.\\

\textbf{Acknowledgements. } Special thanks to Facundo Mémoli for introducing the author to the topic of Gromov-Hausdorff distances and providing guidance while writing this article. The author gratefully acknowledges support from the grants BSF 2020124 and NSF CCF AF 2310412.
\section{Proof of \Cref{34rf43ewewwe}}\label{BigSection}

\begin{lemma}\label{465tgfdewew}
Let $(Y,d_Y)$ be a length space such that, for some constant $C>0$, all loops of diameter $<C$ are contractible in $Y$, and let $y_0\in Y$. Then any two loops $\beta,\beta':[0,1]\to Y$ based at $y_0$ satisfying $d_Y(\beta(t),\beta'(t))<C$ for all $t\in[0,1]$ are path-homotopic.
\end{lemma}

\begin{proof}
Let $\delta>0$ be such that $d_Y(\beta(t),\beta'(t))<C-2\delta$ for all $t\in[0,1]$. 

Let $N\in\mathbb{N}$ be big enough that for all $n=1,\dots,N$, the paths $\beta_n:=\beta|_{\left[\frac{n-1}{N},\frac{n}{N}\right]}$ and $\beta'_n:=\beta'|_{\left[\frac{n-1}{N},\frac{n}{N}\right]}$ have diameter $<\delta$. Note that $\beta=\beta_1\cdots\beta_N$ and $\beta'=\beta'_1\cdots\beta'_N$.  Also, for each $n=1,\dots,N-1$ let $\gamma_n:[0,1]\to Y$ be a path with $\gamma_n(0)=\beta\left(\frac{n}{N}\right),\gamma_n(1)=\beta'\left(\frac{n}{N}\right)$ and with length $<C-2\delta$, as in \Cref{BB'Col}.

\begin{figure}[h]
    \centering
    \includegraphics[width=0.80\linewidth]{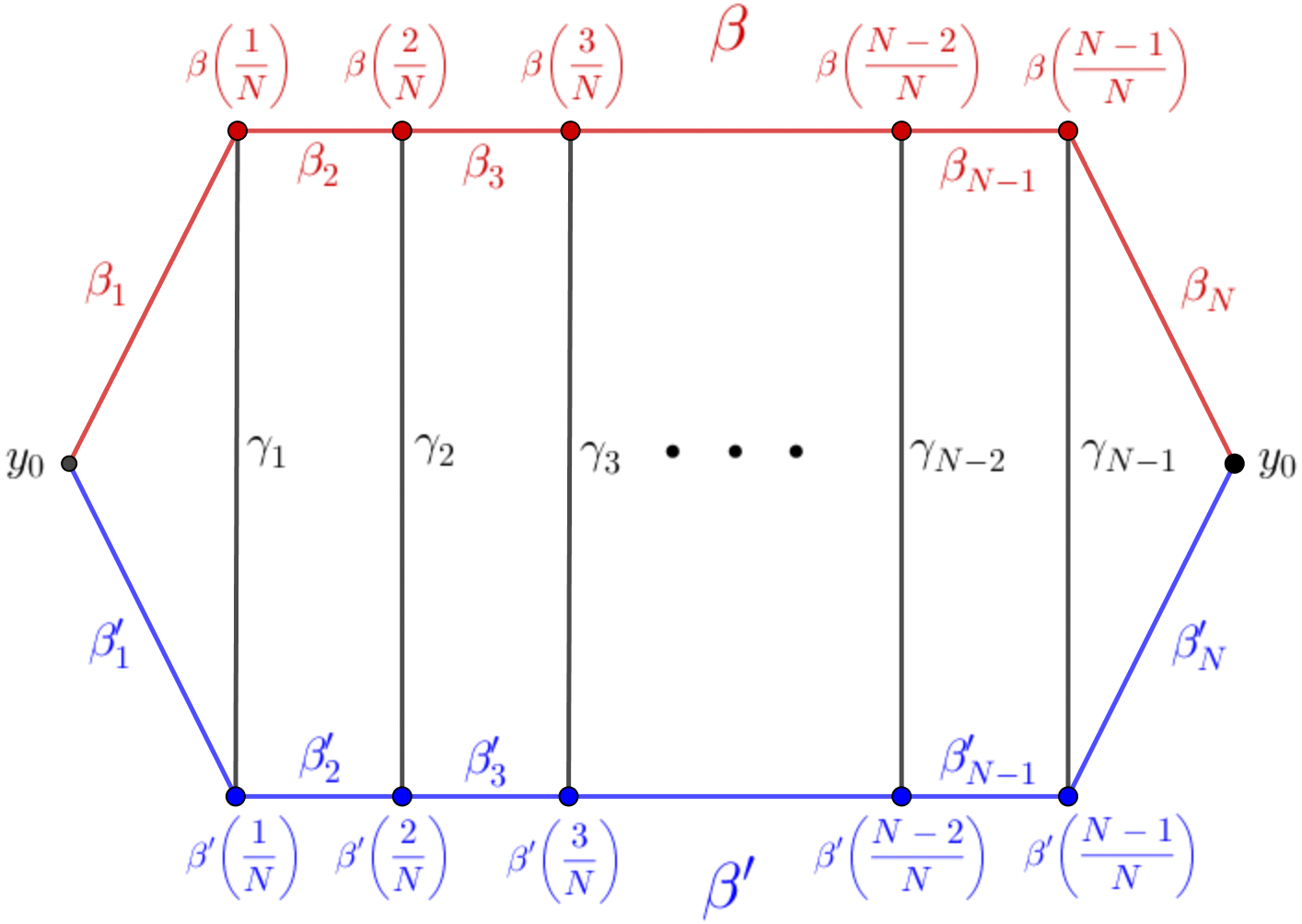}
    \caption{The loops $\beta$ and $\beta'$, with paths $\gamma_1,\dots,\gamma_N$ between them. Note that the $2$ triangles and the $N-2$ rectangles in the figure are all nulhomotopic, as they have diameter $<C$.}
    \label{BB'Col}
\end{figure}

Now consider the sequence of loops $\beta=p_0,p_1,\dots,p_N=\beta'$, where for $n=1,\dots,N-1$ we have $p_n=\beta_1\dots\beta_n\gamma_n\beta'_{n+1}\cdots\beta'_N$. Then for all $n=1,\dots,N$, $p_{n-1}$ is homotopic to $p_n$. Indeed, for $n=2,\dots,N-1$ (the cases $n=1,N$ are easier) we have
\begin{equation*}
p_np_{n-1}^{-1}=(\beta_1\cdots\beta_{n-1})
\beta_n\gamma_n(\beta'_n)^{-1}(\gamma_{n-1})^{-1}(\beta_1\cdots\beta_{n-1})^{-1},
\end{equation*}
and the loop $\beta_n\gamma_n(\beta'_n)^{-1}(\gamma_{n-1})^{-1}$ is null because it has diameter $<C$ (this follows from the facts that $\beta_n,\beta_{n-1}$ have diameter $<\delta$ and $\gamma_n,\gamma_{n-1}$ have length $<C-2\delta$). We conclude then that $\beta$ and $\beta'$ are homotopic, as we wanted.
\end{proof}

\begin{theorem}\label{45trfggff}
Let $(X,d_X),(Y,d_Y)$ be length spaces and let $D>d_{\textup{GH}}(X,Y)$. Suppose $X$ and $Y$ are embedded isometrically into a metric space $(Z,d_Z)$ with $d_{\mathrm{H}}^Z(X,Y)<D$. Let $x_0\in X,y_0\in Y$ with $d_Z(x_0,y_0)<D$. Then for any loop $\alpha:[0,1]\to X$ based at $x_0$ there is a loop $\beta:[0,1]\to Y$ based at $y_0$ such that $d_Z(\alpha(t),\beta(t))<2D$ for all $t\in[0,1]$.

Moreover, if all loops (not necessarily based at $y_0$) of diameter $<4D$ in $Y$ are contractible in $Y$, then the function $\alpha\mapsto\beta$ defined above gives a well defined, surjective\footnote{If $X$ is not a length space, then \Cref{45trfggff} still holds with the same proof, except that the map $\Phi$ need not be surjective.} map $\Phi:\pi_1(X,x_0)\to\pi_1(Y,y_0)$.
\end{theorem}

\begin{proof}
Let $\delta>0$ be such that $d_{\text{H}}^Z(X,Y)<D-\delta$. 
Let $N\in\mathbb{N}$ be big enough that diam$\left(\alpha|_{\left[\frac{n-1}{N},\frac{n}{N}\right]}\right)<\delta$ for all $n=1,\dots,N$, and for each $n=0,\dots,N$ 
define $\beta\left(\frac{n}{N}\right)$ to be a point of $Y$ such that $d_Z\left(\alpha\left(\frac{n}{N}\right),\beta\left(\frac{n}{N}\right)\right)<D-\delta$ (one can choose $\beta(0)=\beta(1)=y_0$). By the triangular inequality, $d_Y\left(\beta\left(\frac{n-1}{N}\right),\beta\left(\frac{n}{N}\right)\right)<2(D-\delta)+\delta<2D$,
\begin{comment}
\begin{equation}
d_Y\left(\beta\left(\frac{n-1}{N}\right),\beta\left(\frac{n}{N}\right)\right)\leq 
d_Z\left(\beta\left(\frac{n-1}{N}\right),\alpha\left(\frac{n-1}{N}\right)\right)+
d_X\left(\alpha\left(\frac{n-1}{N}\right),\alpha\left(\frac{n}{N}\right)\right)+
d_Z\left(\alpha\left(\frac{n}{N}\right),\beta\left(\frac{n}{N}\right)\right)
<2(D-\delta)+\delta=2D-\delta.
\end{equation}
\end{comment}
thus we can define $\beta|_{\left[\frac{n-1}{N},\frac{n}{N}\right]}$ to be a path from $\beta\left(\frac{n-1}{N}\right)$ to $\beta\left(\frac{n}{N}\right)$ of length $<2D$.

Note that all points in $\beta|_{\left[\frac{n-1}{N},\frac{n}{N}\right]}$ are at distance $<D$ of either $\beta\left(\frac{n-1}{N}\right)$ or $\beta\left(\frac{n}{N}\right)$. Moreover, $\beta\left(\frac{n-1}{N}\right)$ and $\beta\left(\frac{n}{N}\right)$ are both at distance $<D$ from all points in $\alpha|_{\left[\frac{n-1}{N},\frac{n}{N}\right]}$. Thus, all points in $\beta|_{\left[\frac{n-1}{N},\frac{n}{N}\right]}$ are at distance $<D+(D-\delta)$ from all points in $\alpha|_{\left[\frac{n-1}{N},\frac{n}{N}\right]}$, and in particular $d_Z(\beta(t),\alpha(t))<2D$ for all $t\in[0,1]$, as we wanted. This concludes the proof of the first part of the theorem.\\

Suppose now that all loops of diameter $<4D$ are contractible in $Y$. Suppose that $\alpha,\alpha':[0,1]\to X$ are path-homotopic loops (the path-homotopy being inside $X$) based at $x_0$, and let $\beta,\beta':[0,1]\to Y$ be loops based at $y_0$ such that $d_Z(\alpha(t),\beta(t))<2D$ and $d_Z(\alpha'(t),\beta'(t))<2D$ for all $t\in[0,1]$. We want to prove that $\beta,\beta'$ are homotopic.

To do it, let $\delta>0$ be such that $d_{\text{H}}^Z(X,Y)<D-\delta$ and such that $d_Z(\alpha(t),\beta(t))<2D-\delta$ and $d_Z(\alpha'(t),\beta'(t))<2D-\delta$ for all $t\in[0,1]$. Let $(\alpha_t)_{t\in[0,1]}$ be a path-homotopy in $X$ from $\alpha_0=\alpha$ to $\alpha_1=\alpha'$. 
By uniform continuity of the homotopy, we can choose big enough $N$ such that for all $n=1,\dots,N$ and for all $t\in[0,1]$ we have $d_X\left(\alpha_{\frac{n-1}{N}}(t),\alpha_{\frac{n}{N}}(t)\right)<\delta$. 
Now by the first part of the theorem, we can let $\beta_{\frac{n}{N}}$ be a loop in $Y$ based at $y_0$ such that $d_Z(\alpha_{\frac{n}{N}}(t),\beta_{\frac{n}{N}}(t))<2(D-\delta)$ for all $t\in[0,1]$ (with $\beta_0=\beta$ and $\beta_1=\beta'$). 
By the triangular inequality we have for all $n=1,\dots,N$ and $t\in[0,1]$ that $d_Z\left(\beta_{\frac{n-1}{N}}(t),\beta_{\frac{n}{N}}(t)\right)\leq4(D-\delta)+\delta<4D$, thus by \Cref{465tgfdewew}, $\beta_{\frac{n-1}{N}}$ and $\beta_{\frac{n}{N}}$ are path-homotopic. We conclude then that $\beta_0=\beta$ and $\beta_1=\beta'$ are path-homotopic. This proves that the map $\Phi:\pi_1(X,x_0)\to\pi_1(Y,y_0)$ is well defined.

To check that $\Phi$ is surjective note that, by the first part of the theorem, for any loop $\beta:[0,1]\to Y$ based at $y_0$ we can find a loop $\alpha:[0,1]\to X$ based at $x_0$ such that $d_Z(\beta(t),\alpha(t))<2D$ for all $t$, so $\Phi([\alpha])=[\beta]$.
\end{proof}

\begin{proof}[Proof of \Cref{34rf43ewewwe}]
Let $(Z,d_Z)$ contain $X,Y$ isometrically and choose basepoints $x_0,y_0$ in $X,Y$ respectively with $d_Z(x_0,y_0)<D$. Define maps $\Phi:\pi_1(X,x_0)\to\pi_1(Y,y_0)$ and $\Psi:\pi_1(Y,y_0)\to\pi_1(X,x_0)$ as in \Cref{45trfggff}. Then the composition $\Psi\circ\Phi$ is the identity (similarly with $\Phi\circ\Psi$): indeed, for any loop $\alpha:[0,1]\to X$ based at $x_0$, $\Phi([\alpha])$ will be $[\beta]$ for some loop $\beta:[0,1]\to Y$ based at $y_0$ such that $d_Z(\alpha(t),\beta(t))<2D$ for all $t\in[0,1]$. But then we also have $\Psi([\beta])=[\alpha]$, so $\Psi\circ\Phi([\alpha])=[\alpha]$.
\end{proof}

\begin{note}
One can use the ideas from this section to find lower bounds for the distance from any simply connected length space to $\mathbb{S}^1$ with the Euclidean metric $d_{\mathbb{R}^2}$ inherited from $\mathbb{R}^2$. 
In $(\mathbb{S}^1,d_{\mathbb{R}^2})$, the conclusion of \Cref{465tgfdewew} holds for $C=2$, that is, any two loops $\beta_0,\beta_1:[0,1]\to\mathbb{S}^1$ with $d_{\mathbb{R}^2}(\beta_0(t),\beta_1(t))<2$ for all $t$ are path-homotopic, with a homotopy $(\beta_s)_{s\in[0,1]}$ between them defined by, for each $t$, letting $s\mapsto\beta_s(t)$ be the shortest geodesic from $\beta_0(t)$ to $\beta_1(t)$. 
Also note that for any $D>0$ and any two points $x,y\in\mathbb{S}^1$ with $d_{\mathbb{R}^2}(x,y)\leq2D$, the points in the shortest path from $x$ to $y$ will not be at distance $>\sqrt{2-2\sqrt{1-D^2}}$ from the set $\{x,y\}$. That and the proof of \Cref{45trfggff} imply that, if we have a length space $X$ at Hausdorff distance $<D$ from $\mathbb{S}^1$ inside some bigger metric space $(Z,d_Z)$, then for each loop $\alpha:[0,1]\to X$ we can find a loop $\beta:[0,1]\to\mathbb{S}^1$ such that for all $t$, $d_Z(\alpha(t),\beta(t))<D+\sqrt{2-2\sqrt{1-D^2}}$. 
If $2\left(D+\sqrt{2-2\sqrt{1-D^2}}\right)<C=2$, as in \Cref{45trfggff} we obtain a well defined surjective map $\pi_1(X)\to\pi_1\left(\mathbb{S}^1\right)$, so $X$ cannot be simply connected.

This proves that the GH distance $D$ from $(\mathbb{S}^1,d_{\mathbb{R}^2})$ to a simply connected length space cannot satisfy $D+\sqrt{2-2\sqrt{1-D^2}}<1$. Thus we obtain a lower bound for $D$ of approximately $0.49165$. We could not construct an optimal example similar to the one in \Cref{EO} and we currently see no reason why this lower bound should be optimal, which naturally leads to the following question:
\end{note}

%\red{More bounds for distances from euclidean $\mathbb{S}^1$ to other spheres?}

\begin{question}
What is the infimal distance from a simply connected geodesic space to $(\mathbb{S}^1,d_{\mathbb{R}^2})$?
\end{question}

\section{A simply connected geodesic space $\mathrm{E}$ with\\ d$_{\textup{GH}}(\mathrm{E},\mathbb{S}^1)=\frac{\pi}{4}$}\label{EO}

\Cref{ES1} shows two metric graphs (see \cite[Definition 3.2.9.]{BBI} for a definition of metric graph) which we will call $\textup{E}$ and $\mathbb{S}^1$ for obvious reasons, with all edges having length $\frac{\pi}{4}$. The graph E has five edges, $a,b,c,d$ and $e$. We have also labelled the eight edges of $\mathbb{S}^1$ with the letters $a,b,c,d,e$. 

We define a function $\Phi:\mathbb{S}^1\to \textup{E}$ which sends each of the $8$ edges of $\mathbb{S}^1$ isometrically to one of the $5$ segments $a,b,c,d,e$ of $\textup{E}$ as indicated by the labels and orientations in the figure. This specification of $\Phi$ gives two possible values $P,Q\in\textup{E}$ to $\Phi(1)$; for definiteness we choose $\Phi(1)=P$. Note that $\Phi$ is locally $1$-Lipschitz except at $1$, where there is a jump of length $\frac{\pi}{2}$ from $Q$ to $P$.
\begin{figure}[h]
    \centering
    \includegraphics[width=0.6\linewidth]{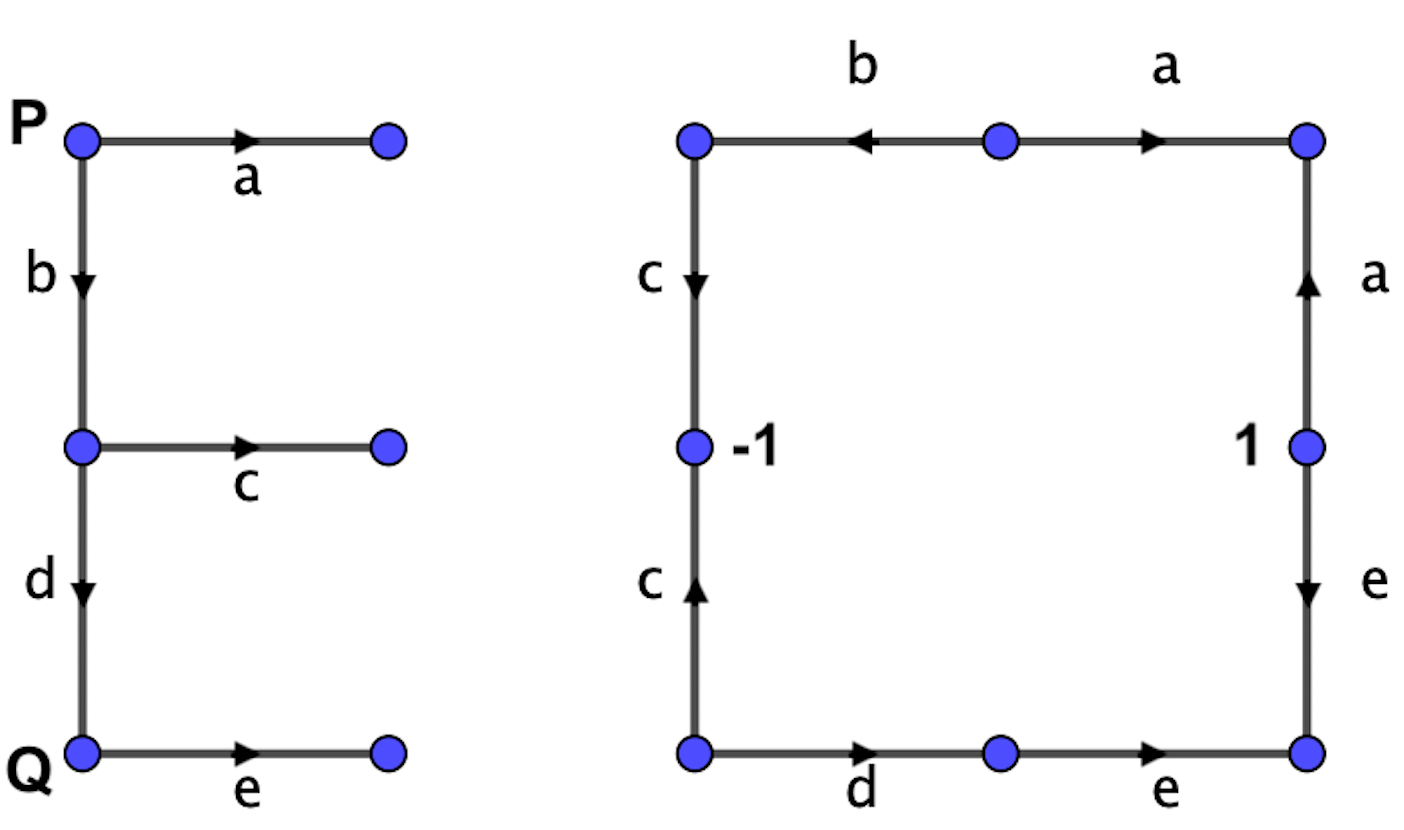}
    \caption{The spaces \textup{E} and $\mathbb{S}^1$.}
    \label{ES1}
\end{figure}

\begin{claim}\label{354ref34r}
The graph of $\Phi$, $R:=\{(x,\Phi(x));x\in\mathbb{S}^1\}\subseteq\mathbb{S}^1\times\textup{E}$, is a correspondence between $\mathbb{S}^1$ and $\textup{E}$ with $\text{dis}(R)\leq\frac{\pi}{2}$.
\end{claim}

Note that, by \Cref{distortandGH}, the claim above implies that $d_{\textup{GH}}(\textup{E},\mathbb{S}^1)\leq \frac{\pi}{4}$. Since $\textup{E}$ is a simply connected geodesic space, by \Cref{BigTheorem} we obtain $d_{\textup{GH}}(\textup{E},\mathbb{S}^1)= \frac{\pi}{4}$, which is the goal of this section.

\begin{proof}[Proof of \Cref{354ref34r}]
First note that $\Phi$ is surjective, so $\Phi$ is indeed a correspondence between $\mathbb{S}^1$ and $E$. It remains to prove that for all $x,x'\in\mathbb{S}^1$ we have \[\big|d_\textup{E}(\Phi(x),\Phi(x'))-d_{\mathbb{S}^1}(x,x')\big|\leq\frac{\pi}{2}.\]

First we prove the inequality $d_\textup{E}(\Phi(x),\Phi(x'))-d_{\mathbb{S}^1}(x,x')\leq\frac{\pi}{2}$. To do this, we may assume $x\neq x'$ and consider the geodesic segment $[x,x']$, given by a unit speed curve $\gamma:[0,d]\to\mathbb{S}^1$ with $\gamma(0)=x$, $\gamma(1)=x'$ and $d:=d_{\mathbb{S}^1}(x,x')\leq\frac{\pi}{2}$. Then the curve $\Phi\circ\gamma:[0,d]\to \textup{E}$ either is $1$-Lipschitz, in which case $d_\textup{E}(\Phi(x),\Phi(x'))\leq\text{len}(\gamma)\leq d$, or it has a jump of length $\frac{\pi}{2}$ at some point $s\in[0,d]$. In this case, $\Phi\circ\gamma$ is still $1$-Lipschitz in $[0,s)$ and $(s,d]$, so we have the inequality we wanted:
\[
d_\textup{E}(\Phi(x),\Phi(x'))\leq\text{length}(\gamma)\leq s+\frac{\pi}{2}+(s-d)=\frac{\pi}{2}+d=\frac{\pi}{2}+d_{\mathbb{S}^1}(x,x').
\]

Now we prove the remaining inequality, $d_{\mathbb{S}^1}(x,x')-d_\textup{E}(\Phi(x),\Phi(x'))\leq\frac{\pi}{2}$. When $x,x'$ are antipodal, we have equality: $d_{\textup{E}}(\Phi(x),\Phi(-x))=\frac{\pi}{2}$. This can be checked using the definition of $\Phi$ in the eight edges of $\mathbb{S}^1$.

For general points $x\neq x'$, note that $1$ cannot be in both of the geodesic segments $[-x',x]$ and $[x',-x]$, because they are antipodal segments of length $<\pi$. Suppose without loss of generality that $1$ is not in $[-x',x]$, so that, if $
\gamma:[0,d]\to\mathbb{S}^1$ is the unit speed geodesic with $\gamma(0)=x'$, $\gamma(d)=-x$ and $d=d_{\mathbb{S}^1}(x',-x)$, then the path $\Phi\circ\gamma$ is $1$-Lipschitz. Now consider the function 
\[
h:[0,d]\to\mathbb{R};h(t):=d_{\mathbb{S}^1}(x,\gamma(t))-d_\textup{E}(\Phi(x),\Phi(\gamma(t))).
\]
Then $h$ is increasing, because $\frac{d}{dt}d_{\mathbb{S}^1}(x,\gamma(t))=1$ and the map $t\mapsto d_\textup{E}(\Phi(x),\Phi(\gamma(t)))$ is $1$-Lipschitz. Thus, as we wanted, 
\[
\frac{\pi}{2}=d_{\mathbb{S}^1}(x,-x)-d_\textup{E}(\Phi(x),\Phi(-x))=\Phi(d)\geq \Phi(0)=d_{\mathbb{S}^1}(x,x')-d_\textup{E}(\Phi(x),\Phi(x')).
\]
\end{proof}

\section{A uniqueness result for the space $\mathrm{E}$}\label{UniquenessSection}
It is natural to ask whether $\mathrm{E}$ is, up to isometry, the only geodesic simply connected space with GH distance $\frac{\pi}{4}$ to $\mathbb{S}^1$. We now use $\mathrm{E}$ to obtain a big family of simply connected geodesic spaces which are exactly at GH distance $\frac{\pi}{4}$ from $\mathbb{S}^1$.
\begin{prop}\label{54rfe443rtgg}
Let $(X,d_X)$, $(Y,d_Y)$ and $(Z,d_Z)$ be nonempty metric spaces, and suppose $\text{diam}(Z)\leq2d_{\textup{GH}}(X,Y)$. Then the space $X\times Z$, with the metric 
\begin{equation*}
d_{X\times Z}((x,z),(x',z')):=\max(d_X(x,x'),d_Z(z,z')), 
\end{equation*}
satisfies $d_{\mathrm{GH}}(X\times Z,Y)\leq d_{\textup{GH}}(X,Y)$.
\end{prop}

\begin{proof}
For any correspondence $R\subseteq X\times Y$, which by \Cref{distortandGH} satisfies dis$(R)\geq2d_{\textup{GH}}(X,Y)\geq\text{diam}(Z)$, we can define the correspondence
\begin{equation*}
R_Z:=\{((x,z),y);(x,y)\in R,z\in Z\}\subseteq(X\times Z)\times Y.
\end{equation*}
We will be done by \Cref{distortandGH} if we prove that $\text{dis}(R_Z)\leq\text{dis}(R)$. Indeed, letting $((x,z),y),((x',z'),y')\in R_Z$, we have
\begin{align*}
d_Y(y,y')-d_{X\times Z}((x,z),(x',z'))
&\leq d_Y(y,y')-d_X(x,x')\\
&\leq\text{dis}(R).\\
d_{X\times Z}((x,z),(x',z'))-d_Y(y,y')
&=\max(d_X(x,x')-d_Y(y,y'),d_Z(z,z')-d_Y(y,y'))\\
&\leq\max(\text{dis}(R),\text{diam}(Z))\\
&\leq\text{dis}(R).\qedhere
\end{align*}
\end{proof}

\begin{cor}\label{453retrgfgdff}
Let $(X,d_X)$ be a nonempty simply connected geodesic space with $\text{diam}(X)\leq\frac{\pi}{2}$. Give $X\times\mathrm{E}$ the metric 
\begin{equation*}
d_{X\times\mathrm{E}}((x,e),(x',e'))=\max(d_X(x,x'),d_{\mathrm{E}}(e,e')).
\end{equation*} 
Then $(X\times\mathrm{E},d_{X\times\mathrm{E}})$ is a simply connected geodesic space and $d_{\textup{GH}}(X\times\mathrm{E})=\frac{\pi}{4}$.
\end{cor}

\begin{proof}
The space $X\times\mathrm{E}$ is simply connected and geodesic (see \cite[Page 2]{Kı}). It then follows from \Cref{54rfe443rtgg} and \Cref{BigTheorem} that $d_{\textup{GH}}(X\times\mathrm{E},\mathbb{S}^1)=\frac{\pi}{4}$.
\end{proof}

\begin{note}\label{45rteretreggfd}
More generally, the proof of \Cref{453retrgfgdff} shows that any geodesic, simply connected subspace $Z\subseteq X\times\mathrm{E}$ such that the coordinate projection $\pi_{\mathrm{E}}:Z\to\mathrm{E}$ is surjective will also be at distance $\frac{\pi}{4}$ from $\mathbb{S}^1$.
\end{note} 

So E is not the only geodesic, simply connected space at distance $\frac{\pi}{4}$ from $\mathbb{S}^1$. However, we prove below that E is the only \textit{complete tree} with minimal length at distance $\frac{\pi}{4}$ from $\mathbb{S}^1$; to formalize this statement we will need to introduce some concepts and notation:

\begin{definition}[{\cite[Defs. 2.1, 2.2.]{Be}}]
An \textit{arc} in a metric space $X$ is a subspace of $X$ homeomorphic to the interval $[0,1]$.
A metric space $(X,d_X)$ is an \textit{$\mathbb{R}$-tree} if for every $x,x'\in X$ there is a unique arc with endpoints $x,y$ and this arc is a geodesic segment (i.e., isometric to an interval of $\mathbb{R}$). We denote by $[x,x']\subseteq X$ the geodesic segment from $x$ to $x'$. 
We will define the \textit{length} of the $\mathbb{R}$-tree $X$ to be the supremum of all sums of lengths of disjoint arcs contained in $X$.
\end{definition}
Note that all $\mathbb{R}$-trees are geodesic by definition, and the distance between any two points $x,x'$ of an $\mathbb{R}$-tree is the length of the arc $[x,x']$. Moreover, any $\mathbb{R}$-tree $X$ is contractible: indeed, for any $x_0\in X$ we can define a homotopy $(f_t)_{t\in[0,1]}$, $f_t:X\to X$, from a constant map to Id$_X$ by letting $f_t(x)$ be the point in $[x,x_0]$ at distance $t\cdot d_X(x,x_0)$ from $x_0$. Thus, by \Cref{BigTheorem}, any $\mathbb{R}$-tree $X$ satisfies $d_{\mathrm{GH}}(X,\mathbb{S}^1)\geq\frac{\pi}{4}$.

\begin{prop}\label{34we23ioepqwpw9}
Let $(X,d_X)$ be a complete $\mathbb{R}$-tree with length $\leq\frac{5\pi}{4}$ and such that $d_{\textup{GH}}(X,\mathbb{S}^1)=\frac{\pi}{4}$. Then $(X,d_X)$ is isometric to \textup{E}.
\end{prop}

\begin{note}\label{Note3efrdfd}
\Cref{34we23ioepqwpw9} does not hold for finite simplicial graphs instead of $\mathbb{R}$-trees: for example, the metric space $\left(\mathbb{S}^1,\frac{d_{\mathbb{S}^1}}{2}\right)$ has length $\pi$ and is at GH distance $\frac{\pi}{4}$ from $\mathbb{S}^1$ (the identity map $\mathbb{S}^1\to\mathbb{S}^1$ gives an optimal correspondence).\\

The condition of having length $\leq\frac{5\pi}{4}$ is also necessary to prove uniqueness of \textup{E} in \Cref{34we23ioepqwpw9}, that is, there are $\mathbb{R}$-trees $X$ not isometric to \textup{E} and such that $d_{\textup{GH}}(X,\mathbb{S}^1)=\frac{\pi}{4}$\footnote{This fact doesn't seem to follow from \Cref{45rteretreggfd}, so we give a separate construction.}. Indeed, consider a metric tree \textup{E'} obtained from E by increasing the length of the edge $c$ (see \Cref{ES1}) from $\frac{\pi}{4}$ to $\frac{\pi}{2}$. We can consider E as a subspace of $\textup{E'}$, with $s:=\textup{E'}\setminus\textup{E}$ being a segment of length $\frac{\pi}{4}$. Letting $R\subseteq\mathbb{S}^1\times\textup{E}$ be the correspondence from \Cref{354ref34r}, we can define a correspondence 
\begin{equation*}
R'=R\bigcup(\{-1\}\times s)\subseteq\mathbb{S}^1\times\textup{E'},
\end{equation*}
and it is not hard to check that $\text{dis}(R')=\frac{\pi}{2}$.
\end{note}

\begin{note}
Note that the condition `$d_{\textup{GH}}(X,\mathbb{S}^1)=\frac{\pi}{4}$' can be equivalently expressed in the following way\footnote{To check that both statements are equivalent, one needs to use that $\mathbb{R}$-trees of finite length are compact, see \Cref{45trgdfv4re}.}: there exist isometric copies of $X$ and $\mathbb{S}^1$ inside a bigger metric space $(Z,d_Z)$ such that the two following conditions are satisfied:
\begin{enumerate}
    \item\label{cond1dXs1} $d_Z(x,\mathbb{S}^1)\leq\frac{\pi}{4}$ for all $x\in X$.
    \item\label{cond2dXs1} $d_Z(p,X)\leq\frac{\pi}{4}$ for all $p\in\mathbb{S}^1$.
\end{enumerate}
\end{note}
Thus, \Cref{34we23ioepqwpw9} may be informally expressed as `besides $\mathrm{E}$, no other complete $\mathbb{R}$-tree $X$ of length $\leq\frac{5\pi}{4}$ can satisfy conditions \ref{cond1dXs1} and \ref{cond2dXs1} simultaneously'. If we remove condition \ref{cond2dXs1}, it is clear that many complete $\mathbb{R}$-trees of length $\leq\frac{5\pi}{4}$ can satisfy condition \ref{cond1dXs1} (e.g. the $\mathbb{R}$-tree formed by a single point). So we certainly need to use condition \ref{cond2dXs1} in order to prove our theorem. However, one may at first think that condition \ref{cond1dXs1} is not necessary to prove our uniqueness theorem, that is, that any $\mathbb{R}$-tree $X$ of length $\leq\frac{5\pi}{4}$ satisfying condition \ref{cond2dXs1} is isometric to $\mathrm{E}$. That is, `if an $\mathbb{R}$-tree $X$ is smaller than $E$, then we cannot place $X$ and $\mathbb{S}^1$ in such a way that all points of $\mathbb{S}^1$ are at distance $\leq\frac{\pi}{4}$ of $X$'. This is, however, not true. Indeed, consider the $\mathbb{R}$-tree $X$ shown below, which has four edges of length $\frac{\pi}{4}$ and four extreme points $p_1,p_2,p_3,p_4$. Then $X$ has length $\pi$, but if we divide $\mathbb{S}^1$ into four intervals $I_1,I_2,I_3,I_4$ of length $\frac{\pi}{2}$, one can give the disjoint union $\mathbb{S}^1\sqcup X$ a metric in which, for each $i=1,2,3,4$, all the points of $I_i$ are at distance $\frac{\pi}{4}$ from $p_i$.
\begin{figure}[h]
    \centering
    \includegraphics[width=0.25\linewidth]{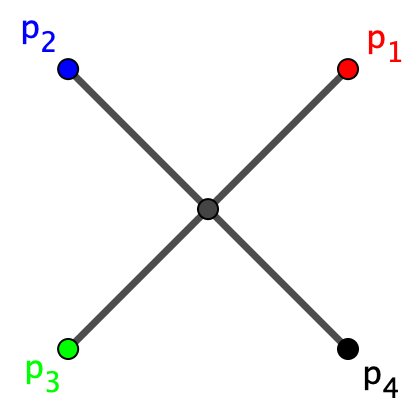}
    \qquad\qquad
    \includegraphics[width=0.25\linewidth]{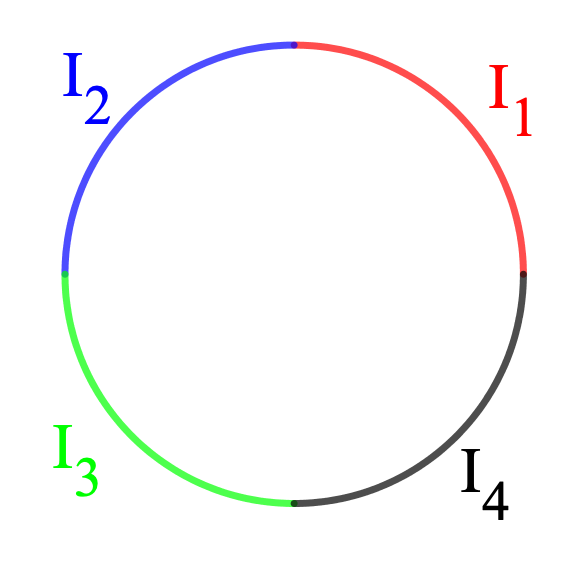}
    %\caption{}
    %\label{ES1}
\end{figure}

The proof of \Cref{34we23ioepqwpw9} comprises the rest of this section. Fix a complete $\mathbb{R}$-tree $(X,d_X)$ of length at most $\frac{5\pi}{4}$ satisfying $d_{\mathrm{GH}}(X,\mathbb{S}^1)=\frac{\pi}{4}$.
\begin{claim}\label{45trgdfv4re}
$X$ is compact.
\end{claim}

\begin{proof}
It is enough to prove that $X$ is totally bounded. If not, there is some $\varepsilon>0$ and points $(x_n)_{n\in\mathbb{N}}\in X$ with $d(x_i,x_j)>4\varepsilon$ if $i\neq j$. So the balls $B_n:=B_X(x_n,2\varepsilon)$ are pairwise disjoint, and each contains a segment of length $\varepsilon$, contradicting the fact that $X$ has finite length.
\end{proof}

As $X$ and $\mathbb{S}^1$ are compact and $d_{\mathrm{GH}}(X,\mathbb{S}^1)=\frac{\pi}{4}$, it is proved in \cite[Proposition 1.1.]{CM} that there is a correspondence $R\subseteq X\times\mathbb{S}^1$ achieving the minimal distortion $\textup{dis}(R)=\frac{\pi}{2}$, and such that $R$ is a closed subspace of $X\times\mathbb{S}^1$.

\begin{comment}
\begin{proof}
The proof only uses that $X$ and $\mathbb{S}^1$ are compact. Give $X\times\mathbb{S}^1$ the metric
\begin{equation*}
d_{X\times\mathbb{S}^1}((x,p),(x',p'))=\max(d_X(x,x'),d_{\mathbb{S}^1}(p,p')).
\end{equation*}
For each $n\in\mathbb{N}$, by \Cref{distortandGH} we have a correspondence $R_n\subseteq X\times\mathbb{S}^1$ with $\textup{dis}(R_n)<\frac{\pi}{2}+\frac{1}{n}$. We will let $R$ be the subset of points $(x,p)\in X\times\mathbb{S}^1$ such that there exists a sequence $(x_n,p_n)_{n\in\mathbb{N}}$ with $(x_n,p_n)\in R_n$ and such that $(x,p)$ is a cluster point of the sequence $(x_n,p_n)_n$.

It follows from this definition that $R$ is closed in $X\times\mathbb{S}^1$ and $\text{dis}(R)\leq\frac{\pi}{2}$. Suppose for contradiction that $X$ is not a correspondence, so that there is some point $x\in X$ (or $p\in\mathbb{S}^1$; a similar argument works in that case) such that $(\{x\}\times\mathbb{S}^1)\cap R=\varnothing$. 
As $\mathbb{S}^1$ is compact, we have $\varepsilon:=\frac{1}{2}d_{X\times\mathbb{S}^1}(\{x\}\times\mathbb{S}^1,R)>0$, so that the compact set $\overline{B}(x,2\varepsilon)\times\mathbb{S}^1$ does not intersect $R$. 
However, for each $n$ there is some point $(x_n,p_n)\in R_n\cap\left(\overline{B}(x,\varepsilon)\times\mathbb{S}^1\right)$, and taking a cluster point of the sequence $(x_n,p_n)_n$ gives us a point in $R\cap\left(\overline{B}(x,\varepsilon)\times\mathbb{S}^1\right)$, a contradiction.
\end{proof}
\end{comment}

So from now on, we fix a closed correspondence $R\subseteq X\times\mathbb{S}^1$ with distortion $\frac{\pi}{2}$. For any $x\in X$ we let $R[x]:=\{p\in\mathbb{S}^1;(x,p)\in R\}$, and for any $p\in\mathbb{S}^1$ we let $R[p]:=\{x\in X;(x,p)\in R\}$ and we denote by $-p$ the antipodal of $p$. We will fix two points $x_-,x_+\in X$ with $d_X(x_-,x_+)=\text{diam}(X)$ and let $x_0$ be the midpoint of $[x_-,x_+]$. Finally, we define the `projection' map 
\begin{equation*}
\Pi:X\to[x_-,x_+]
\end{equation*}
by the equation $d_X(x,\Pi(x))=d_X(x,[x_-,x_+])$.
Note that if two points $x,x'\in X$ have $\Pi(x)\neq\Pi(x')$, then the union of $[x,\Pi(x)],[\Pi(x),\Pi(x')]$ and $[\Pi(x'),x']$ is the geodesic arc from $x$ to $x'$ (see figure \Cref{ES342ewdfgdfv}), so we have
\begin{equation}\label{43frreefseewrewrerwwer}
d_X(x,x')=d_X(x,\Pi(x))+d_X(\Pi(x),\Pi(x'))+d_X(\Pi(x'),x').
\end{equation}
\begin{figure}[h]
    \centering
    \includegraphics[width=0.6\linewidth]{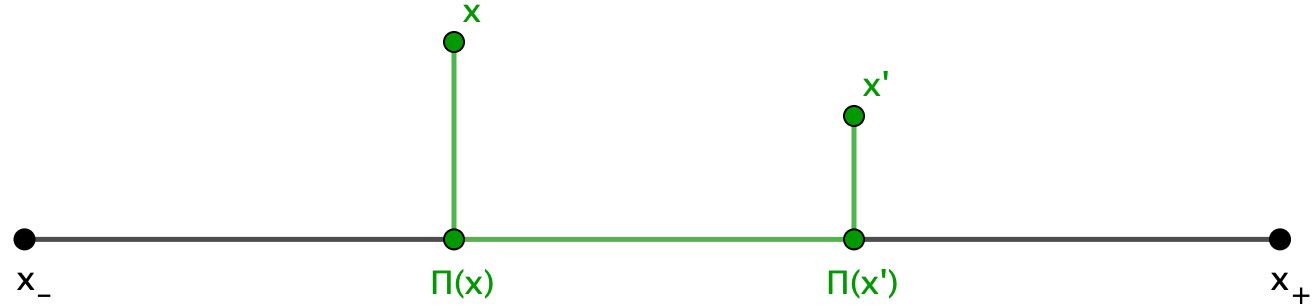}
    \caption{Geodesics between points $x,x'\in X$ with $\Pi(x)\neq\Pi(x')$.}
    \label{ES342ewdfgdfv}
\end{figure}

We will repeatedly use \Cref{43frreefseewrewrerwwer} to compute distances in $X$.

\begin{definition}
A \textit{tripod} is an $\mathbb{R}$-tree $Y$ with three points $y_0,y_1,y_2\in Y$ (the \textit{extreme points} of $Y$) such that $Y=[y_0,y_1]\cup[y_0,y_2]\cup[y_1,y_2]$.
\end{definition}

\begin{claim}\label{45trefd43}
We have $d_X(x_-,x_+)=\text{diam}(X)=\pi$, and the space $X$ is a tripod with extreme points $x_-,x_+$, and a point $x_1$ with $d_X(x_1,\Pi(x_1))=\frac{\pi}{4}$.
\end{claim}

\begin{proof}
We first deduce that $\text{diam}(X)\geq\pi$. To do it, note that any point $x\in X$ is at distance $\leq\frac{\text{diam}(X)}{2}$ from $x_0$: if not, depending on whether $\Pi(x)\in[x_-,x_0]$ or $\Pi(x)\in[x_0,x_+]$ we would have $d_X(x,x_+)>\text{diam}(X)$ or $d_X(x,x_-)>\text{diam}(X)$, a contradiction. Now let $p_0\in R[x_0]$ and let $y\in R[-p_0]$. Using that $\text{dis}(R)\leq\frac{\pi}{2}$, we obtain
\begin{equation*}
\text{diam}(X)
\geq2d_X(x_0,y)
\geq2\left(d_{\mathbb{S}^1}(p_0,-p_0)
-\frac{\pi}{2}\right)
=\pi.
\end{equation*}
We will now deduce the rest of the theorem. First note that for all $x\in X$ we have $d_X(x,\Pi(x))\leq\frac{\pi}{4}$; indeed, as $X$ has length $\leq\frac{5\pi}{4}$, we have
\begin{equation*}
d_X(x,\Pi(x))=\text{length}([x,\Pi(x)])\leq\text{length}(X)-\text{length}([x_-,x_+])\leq\frac{5\pi}{4}-\pi\leq\frac{\pi}{4}.
\end{equation*}
And if $d_X(x,\Pi(x))=\frac{\pi}{4}$ for some $x\in X$, then all the inequalities above are equalities, so $d_X(x_-,x_+)=\pi$ and $\text{length}(X)=\text{length}([x_-,x_+])+\text{length}([x,\Pi(x)])$, so the tree $X$ must be the union $[x_-,x_+]\cup[x,\Pi(x)]$ and we would be done proving \Cref{45trefd43}. We can therefore suppose $d_X(x,\Pi(x))<\frac{\pi}{4}$ for all $x\in X$, and we will obtain a contradiction.

For each $p\in\mathbb{S}^1$ choose a point $x_p\in R[p]$. Note that we have $d_X(x_p,x_{-p})\geq d_{\mathbb{S}^1}(p,-p)-\frac{\pi}{2}=\frac{\pi}{2}$ for all $p\in\mathbb{S}^1$. Moreover, $\Pi(x_p)\neq\Pi(x_{-p})$ for all $p$, because
\begin{equation}\label{43trefrfd}
d_X(\Pi(x_{-p}),\Pi(x_{-p}))\geq d_X(x_p,x_{-p})-d_X(x_p,\Pi(x_p))-d_X(x_{-p},\Pi(x_{-p}))>\frac{\pi}{2}-\frac{\pi}{4}-\frac{\pi}{4}=0.
\end{equation}
So, if we order the segment $[x_-,x_+]$ by $x>y$ iff $d_X(x_-,x)>d_X(x_-,y)$, then for any $p\in\mathbb{S}^1$ we have either $\Pi(x_p)>\Pi(x_{-p})$ or $\Pi(x_p)<\Pi(x_{-p})$.

Now, let $(p_t)_{t\in[0,\pi]}$ be a geodesic segment of length $\pi$ in $\mathbb{S}^1$, so that $p_\pi=-p_0$. We can suppose w.l.o.g. that $\Pi(x_{p_0})<\Pi(x_{-p_0})$. Thus, $\Pi(x_{p_\pi})>\Pi(x_{-p_\pi})$.

Let $s:=\inf\{t\in[0,\pi];\Pi(x_{p_0})>\Pi(x_{-p_0})\}\in[0,\pi]$. Suppose that $\Pi(x_{p_s})>\Pi(x_{-p_s})$ (the case $\Pi(x_{p_s})<\Pi(x_{-p_s})$ is similar). Then by definition of $s$, there has to be some sequence $(q_n)_n$ in $\mathbb{S}^1$ convergent to $p_s$ and such that $\Pi(x_{q_n})<\Pi(x_{-q_n})$ for all $n$. By taking a subsequence, we can assume that $(x_{q_n})_n,(x_{-q_n})_n$ converge when $n\to\infty$ to some points $x_\infty,x_{-\infty}\in X$ respectively, so that $\Pi(x_\infty)\leq\Pi(x_{-\infty})$ and, as $R\subseteq X\times\mathbb{S}^1$ is closed, we have $x_\infty\in R[p_s]$ and $x_{-\infty}\in R[-p_s]$. 

So we have points $x_{p_s},x_\infty\in R[p_s]$ and $x_{-p_s},x_{-\infty}\in R[-p_s]$ with $\Pi(x_{p_s})\geq\Pi(x_{-p_s})$ and $\Pi(x_\infty)\leq\Pi(x_{-\infty})$. Thus, we cannot have that both $\Pi(x_{p_s})$ and $\Pi(x_{\infty})$ are strictly smaller than both $\Pi(x_{-p_s})$ and $\Pi(x_{-\infty})$ or viceversa. Equivalently, there are points $x,y,x'\in\{x_{p_s},x_{-p_s},x_\infty,x_{-\infty}\}$ such that $\Pi(x)\leq\Pi(y)\leq\Pi(x')$ and for some $p\in\{p_s,-p_s\}$ we have $x,x'\in R[p]$ and $y\in R[-p]$ (this in turn implies that $\Pi(x)<\Pi(y)<\Pi(x')$: we cannot have equality by \Cref{43trefrfd}). But then, we have a contradiction: as $\text{dis}(R)=\frac{\pi}{2}$,
\begin{multline*}
\frac{\pi}{2}
\geq d_X(x,x')=d_X(x,\Pi(x))+d_X(\Pi(x),\Pi(y))
+d_X(\Pi(y),\Pi(x'))+d_X(\Pi(x'),x')\\
=d_X(x,y)+d_X(y,x')-2d(y,\Pi(y))
>\frac{\pi}{2}+\frac{\pi}{2}-2\frac{\pi}{4}
=\frac{\pi}{2}.\qedhere
\end{multline*}
\end{proof}

Now, let $[x_1,x_2]$ be the segment of length $\frac{\pi}{4}$ we append to $[x_-,x_+]$ to obtain $X$, with $x_2\in[x_-,x_+]$. We necessarily have $d_X(x_0,x_2)\leq\frac{\pi}{4}$ (as $\text{diam}(X)=\pi$), and we may assume without loss of generality that $d_X(x_-,x_2)\leq d_X(x_+,x_2)$. Thus we already have a very concrete description of $X$, illustrated in \Cref{ES342ew}. Note that $X$ is determined up to isometry by $d_X(x_0,x_2)$, so we will be done if we prove that $x_2=x_0$, in which case the space $X$ is isometric to E.

\begin{figure}[h]
    \centering
    \includegraphics[width=0.6\linewidth]{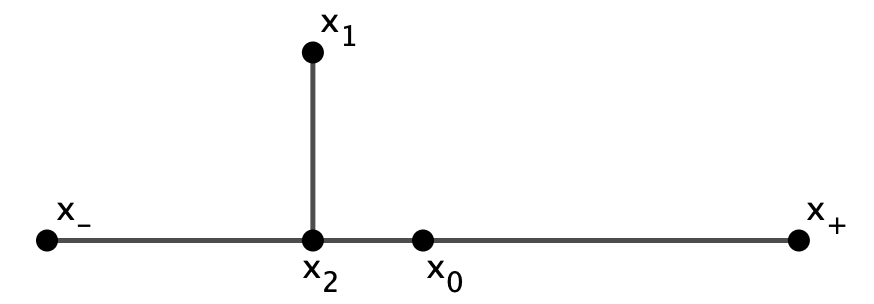}
    \caption{Notation for the tree $X$.}
    \label{ES342ew}
\end{figure}

For each $p\in\mathbb{S}^1$, we will denote by $I_p\subseteq\mathbb{S}^1$ the open interval centered at $p$ with radius $\frac{\pi}{4}$. Also, we divide the tree $X$ into two parts: $X_+:=[x_0,x_+]$ and $X_-=[x_0,x_-]\cup[x_0,x_1]$.

\begin{claim}\label{45tref44343}
Let $p\in R[x_0]$. Then $R[-p]\subseteq\{x_-,x_1,x_+\}$, and exactly one of the following cases happens:
\begin{enumerate}
    \item[\mylabel{45tref44343case1}{Case 1.}] $R[-p]=\{x_+\}$. Then no point of $I_{-p}$ can be related to any point of $X_-$.
    \item[\mylabel{45tref44343case2}{Case 2.}] $R[-p]\subseteq\{x_-,x_1\}$. Then no point of $I_{-p}$ can be related to a point of $X_+$.
\end{enumerate}
\end{claim}

\begin{proof}
There has to be some point $y$ in $R[-p]$, because $R$ is a correspondence. As $R$ has distortion $\frac{\pi}{2}$ and $d_{\mathbb{S}^1}(p,-p)=\pi$, we have $d_X(x_0,y)\geq\frac{\pi}{2}$, which implies $y\in\{x_-,x_1,x_+\}$ (and $y=x_1$ may only happen if $d_X(x_0,x_2)=\frac{\pi}{4}$).

If $y=x_+$, then for any $p_1\in I_{-p}$ and $x\in R[p_1]$ we have:
\begin{align*}
d_X(x_0,x)&\geq d_{\mathbb{S}^1}(p,p_1)-\frac{\pi}{2}>\frac{\pi}{4}.\\
d_X(x_+,x)&\leq d_{\mathbb{S}^1}(-p,p_1)+\frac{\pi}{2}<\frac{3\pi}{4}.
\end{align*}
But no point $x\in X_-$ can satisfy both of these inequalities at the same time, thus, as we wanted, no point of $I_{-p}$ can be related to any point of $X_-$. The cases $y=x_1$ and $y=x_-$ are similar. Also note that if $x_+\in R[p]$, $x_1$ or $x_-$ cannot be in $R[p]$, because they are at distance $>\frac{\pi}{2}$ from $x_+$ and $\textup{dis}(R)=\frac{\pi}{2}$.
\end{proof}

\begin{claim}\label{54tergfd565445re}
We have $R[x_0]=\{p_0,p_0'\}$ for some $p_0,p_0'\in\mathbb{S}^1$ with $d_{\mathbb{S}^1}(p_0,p_0')=\frac{\pi}{2}$ and such that $R[-p_0]=\{x_+\}$, $R[-p_0']\subseteq\{x_-,x_1\}$.
\end{claim}

\begin{proof}
We first prove that there are points $p_0,p_0'\in R[x_0]$ in \ref{45tref44343case1} and \ref{45tref44343case2} of \Cref{45tref44343} respectively. Suppose for contradiction that all the points of $R[x_0]$ are in \ref{45tref44343case2} (the same argument works for \ref{45tref44343case1}), so that all points in $\bigcup_{p\in R[x_0]}I_{-p}$ are not related to points of $X_+$. Note that any point $y\in[x_-,x_0]$ with $0<d_X(y,x_0)<0.1$ satisfies $X_-\subseteq B_X\left(y,\frac{\pi}{2}\right)$. So as $\text{dis}(R)\leq\frac{\pi}{2}$, for any $q\in R[y]$ its antipodal $-q$ must be related to some point of $X^+$. Thus $q\not\in\bigcup_{p\in R[x_0]}I_p$. Letting $y$ approach $x_0$, and using that $R\subseteq X\times\mathbb{S}^1$ is closed, we obtain that there is some point of $R[x_0]$ outside $\bigcup_{p\in R[x_0]}I_p$, a contradiction.

So let $p_0,p_0'\in R[x_0]$ satisfy \ref{45tref44343case1} and \ref{45tref44343case2} of \Cref{45tref44343} respectively. Then we necessarily have $d_{\mathbb{S}^1}(p_0,p_0')\geq\frac{\pi}{2}$; if not, $I_{-p_0}\cap I_{-p_0'}$ would be nonempty, and its points could not be related to either $X_-$ or $X_+$. Now, as $\textup{diam}(R[x_0])\leq\frac{\pi}{2}$, $R[x_0]$ has to be contained in the geodesic segment $[p_0,p_0']$, and all points in $R[x_0]$ have to be at distance $\geq\frac{\pi}{2}$ from either $p_0$ or $p_0'$ for the same reasoning we just stated. So $R[x_0]=\{p_0,p_0'\}$, as we wanted.
\end{proof}

We now conclude the proof of \Cref{34we23ioepqwpw9}. Let $p_0,p_0'$ be as in \Cref{54tergfd565445re} and let $p_3,p_4\in\mathbb{S}^1$ be at the same distance from $p_0,p_0'$, as in the figure below.

\begin{figure}[h]
    \centering
    \includegraphics[width=0.4\linewidth]{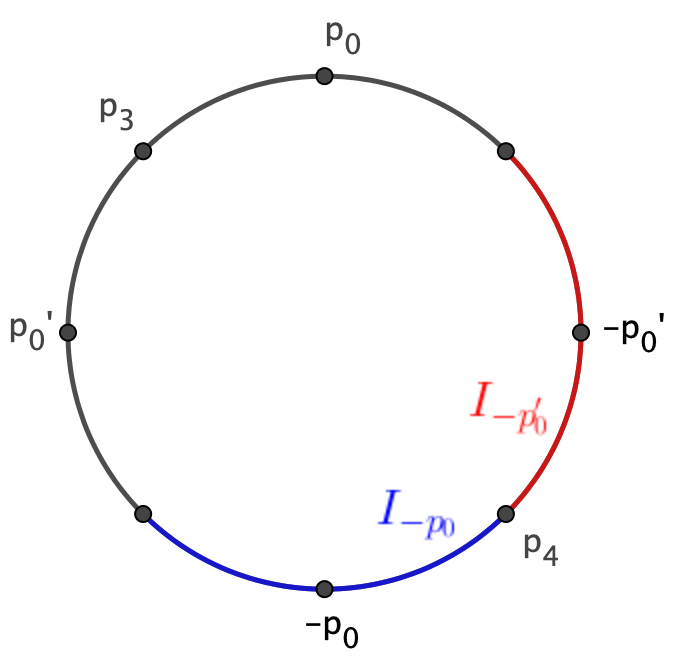}
    %\caption{Notation for the tree $X$.}
    %\label{ES342ew}
\end{figure}

Note that, as $R$ is closed in $X\times\mathbb{S}^1$ and all points of $I_{-p_0}$ are related to some point of $X_+$, $p_4$ has to be related to some point $x_4^+\in X_+$. Similarly, $p_4$ is related to some point $x_4^-\in X_-$. Note that $d_X(x_0,x_4^+)+d_X(x_0,x_4^-)=d_X(x_4^+,x_4^-)$, so using that $\textup{dis}(R)=\frac{\pi}{2}$, we have
\begin{multline*}
\frac{\pi}{2}
=d_{\mathbb{S}^1}(p_4,p_4)+\frac{\pi}{2}
\geq d_X(x_4^+,x_4^-)
=d_X(x_0,x_4^+)+d_X(x_0,x_4^-)\\
\geq\left(d_{\mathbb{S}^1}(p_0,p_4)-\frac{\pi}{2}\right)+\left(d_{\mathbb{S}^1}(p_0,p_4)-\frac{\pi}{2}\right)
=\frac{\pi}{4}+\frac{\pi}{4}=\frac{\pi}{2}.
\end{multline*}
Thus the inequalities above are equalities, and $d_X(x_0,x_4^+)=d_X(x_0,x_4^-)=\frac{\pi}{4}$.

Meaning that $x_4^+$ is the midpoint of $[x_0,x_+]$ and $x_4^-$ is either the midpoint of $[x_0,x_-]$ or some point in $[x_0,x_1]$. Now, let $x_3\in R[p_3]$. As $d_{\mathbb{S}^1}(p_3,p_4)=\pi$ and $\textup{dis}(R)=\frac{\pi}{2}$, we have $d_X(x_3,x_4^+),d_X(x_3,x_4^-)\geq\frac{\pi}{2}$. Now we divide in cases:
\begin{itemize}
    \item $x_4^-$ is the midpoint of $[x_0,x_-]$. Then there are no points in $X$ at distance $\geq\frac{\pi}{2}$ from both $x_4^-$ and  $x_4^+$, unless $x_2=x_0$ and $x_3=x_1$, proving \Cref{34we23ioepqwpw9}.
    \item $x_4^-$ is in $[x_0,x_1]$. If $d_X(x_0,x_2)=\frac{\pi}{4}$, then again $x_4^-=x_2$ is the midpoint of the segment $[x_0,x_-]$, so we are back to the previous case. If $d_X(x_0,x_2)<\frac{\pi}{4}$, then $R[-p_0']=\{x_-\}$ by the proof of \Cref{45tref44343}. In this case we also have $x_2=x_0$: if not, the inequalities $d_X(x_3,x_4^+),d_X(x_3,x_4^-)\geq\frac{\pi}{2}$ would imply that $x_3$ is at distance $<\frac{\pi}{4}$ from $x_-$, which contradicts the facts that $\textup{dis}(R)=\frac{\pi}{2}$ and $d_{\mathbb{S}^1}(p_3,-p_0)=\frac{3\pi}{4}$.
\end{itemize}

\bibliographystyle{amsalpha}

\end{document}